\documentclass[leqno,final]{siamltex}
\usepackage{amssymb}
\usepackage{lineno,hyperref}
\usepackage{graphicx}
\usepackage{color}
\usepackage{pstricks}
\usepackage{amsmath,amstext,amssymb}
\usepackage{multirow}
\newtheorem{remark}{Remark}[section]
\newtheorem{Theorem}{Theorem}[section]
\numberwithin{equation}{section}
\newcommand{\abs}[1]{\left\vert#1\right\vert}

\newcommand{\norm}[1]{\left\Vert#1\right\Vert}
\newcommand{\norml}[2]{\left\Vert#1\right\Vert_{L^2(#2)}}

\newcommand{\av}[1]{\left\{#1\right\}}
\newcommand{\jm}[1]{\left[#1\right]}

\newcommand{\T}{\mathcal{T}}

\newcommand{\bn}{\mathbf{n}}

\newcommand{\nn}{\nonumber}
\newcommand{\ls}{\lesssim}

\newcommand{\be}{\beta}

\newcommand{\ep}{\varepsilon}
\newcommand{\eps}{\varepsilon}
\newcommand{\ga}{\gamma}
\newcommand{\Ga}{\Gamma}

\newcommand{\na}{\nabla}

\newcommand{\Om}{\Omega}
\newcommand{\pa}{\partial}

\newcommand{\dx}{\,\mathrm{d}x}
\newcommand{\dy}{\,\mathrm{d}y}
\newcommand{\ds}{\,\mathrm{d}s}

\title{Multiscale discontinuous Petrov--Galerkin method for the multiscale elliptic problems}
\author{Fei Song\thanks{ Department of Mathematics, Nanjing University, Jiangsu,
210093, P.R. China. Current address: College of Science, Nanjing Forestry University, Jiangsu,
210037, P.R. China. ({\tt songfei@smail.nju.edu.cn}).}
{The work of this author was
partially supported by the University Postgraduate Research and Innovation Project of Jiangsu Province 2014 under Grant KYZZ\_0021.}
\and
Weibing Deng
\thanks{Department of Mathematics, Nanjing University, Jiangsu,
210093, P.R. China. ({\tt wbdeng@nju.edu.cn}). The work of this author was
partially supported by the NSF of China grant 10971096 and the Project Funded by the Priority Academic Program Development of Jiangsu Higher Education Institutions.}
}

\begin{document}

\maketitle


\setcounter{page}{1}
\begin{abstract}
{In this paper we present a new multiscale discontinuous Petrov--Galerkin method  (MsDPGM)  for multiscale elliptic problems. This method utilizes the classical oversampling multiscale basis in the framework of Petrov--Galerkin version of discontinuous Galerkin finite element method, allowing us to better cope with multiscale features in the solution.
The introduced  MsDPGM  takes advantages of the multiscale Petrov--Galerkin method (MsPGM) and discontinuous Galerkin method (DGM), which can eliminate the resonance error completely, and can decrease the computational complexity, allowing for more efficient solution algorithms. Upon the $H^2$ norm error estimate between the multiscale solution and the homogenized solution with the first order corrector, we give a detailed multiscale convergence analysis under the assumption that the oscillating coefficient is periodic. We also investigate the corresponding  multiscale discontinuous finite element method (MsDFEM) which coupling the classical oversampling multiscale basis with DGM since it has not been studied detailedly in both aspects of error analysis and numerical tests in the literature. Numerical experiments  are carried out for the multiscale elliptic problems with periodic and randomly generated log-normal coefficients to demonstrate the proposed method.} 
\end{abstract}



\section{Introduction}
 This paper considers the numerical approximation of second order elliptic problems with heterogeneous and highly oscillating coefficients.
 These problems arise in many applications such as flows in porous media or composite materials. The numerical simulation of such problems
 in heterogeneous media poses major mathematical and computational challenges.
Standard numerical methods such as the finite element method (FEM) or the finite volume method (FVM) usually require the mesh size very fine.
This leads to tremendous amount of computer memory and CPU time.
{In the past several decades, a number of multiscale numerical methods have been proposed to solve these problems,
see e.g., multiscale finite element method (MsFEM) \cite{EHW,HW,HWC}, heterogeneous multiscale method (HMM) \cite{EE1,EE3,EMZ},
upscaling or numerical homogenization method \cite{CDY,dur,Ewing,Farmer}, variational multiscale method (or the residual-free bubble method) \cite{BFHR,BMS,Hu98,Hu,San}, wavelet homogenization techniques \cite{DE,ER}, and multigrid
numerical homogenization techniques \cite{FB2,MDH}.}
Most of them are presented on meshes that are coarser than the scale of oscillations.
The small scale effect on the coarse scale is either captured by localized multiscale basis functions or modeled into the coarse scale equations
with prescribed analytical forms.

In this paper, the framework of the MsFEM is used to propose a new
method. { Two main ingredients of the MsFEM are the global formulation of the method such as various finite element methods and the construction of basis functions. The key point of MsFEM is to construct the multiscale
basis from the local solutions of the elliptic operator for finite
element formulation.} There have been many extensions and applications
of the method in the past fifteen years (cf. \cite{Aarnes2004,AE2008,
AEJ2008,CH2002,cc06,cc07,cc0708,DYX09,EGHE2006,EH2007,HWZh,JLT2003}).
We refer the reader to the book \cite{EH2009} for more discussions
on the theory and applications of MsFEMs.

{ It is shown that the oversampling MsFEM is a nonconforming FEM where the numerical solution has certain continuity across the inner-element boundaries, while its basis functions are discontinuous at the inner-element boundaries (see \cite{HW,EHW}). Note that the discontinuous Garlerkin (DG) FEMs do not ask for any continuity, which comes the natural idea that using the DG FEM as the global formulation coupling with the oversampling multiscale bases (see \cite{EH2009}).
DG FEMs for elliptic boundary value problems have been studied since late 1970s, and it is now an active research area (see \cite{dd76,arnold82,abcm01}).
Examples of the DG methods include the Local Discontinuous Galerkin (LDG)
method \cite{Aarnes2005,CASTILLO2006,P2000An}, and the interior penalty discontinuous Galerkin (IPDG) methods \cite{arnold82,abcm01,b73,bz73,bastian2009unfitted,bh00,cc04}.
In this paper we are concerned with the IPDG method, still named DGM.
DG methods admit good local
conservation properties of the state variable and also offer the use
of very general meshes due to the lack of inter-element continuity
requirement, e.g., meshes that contain several different types of
elements or hanging nodes. These features are crucial in many
multiscale applications (see \cite{DW2014,SDW2015}).}

{
There have been shown several multiscale methods related with DG methods in the past ten years. For instance, a multiscale model reduction technique in the
framework of the DG FEM for the high--contrast problems, named Generalized Multiscale Finite Element method, was presented in \cite{ya2013}. The special multiscale basis
functions of the DG approximation space to capture
the singularity of the solutions were discussed in \cite{shu2011,shu2008,shu2014}. The variational multiscale method based on the DGM for the elliptic multiscale problems without any assumption on scale separation or periodicity were proposed in \cite{Mp131,Mp13}. Heterogeneous multiscale method based on DGM for homogenization or advection--diffusion problems were presented in \cite{aa2008,AH2014}.
However, to our knowledge, the multiscale discontinuous finite element method (MsDFEM) which coupling the classical oversampling multiscale basis with the discontinuous Galerkin method has not been studied detailedly in both aspects of error analysis and numerical tests in the literature. To fill this gap, in this paper we provide the formulation and the corresponding error estimate of the MsDFEM. Our numerical experiments show that MsDFEM takes advantages of MsFEM and DGM, which can eliminate the resonance error and obtain more accurate results than the classical MsFEM. 

Further, we notice that the Petrov--Galerkin (PG)
method can decrease the computational complexity significantly,
allowing for more efficient solution algorithms. Moreover, it has been found that
the MsPGM can eliminate the resonance
error by using the oversampling technique \cite{HW} and the
conforming piecewise linear functions as test functions \cite{HWZh, HWC}.
Therefore, in this paper we try to use the discontinuous oversampling multiscale space in the framework of PG method, which couples both DGM and MsPGM, named multiscale discontinuous Petrov--Galerkin method (MsDPGM).
There are two key issues of the MsDPGM to
consider. The first one is how to define its bilinear form and prove the coercive condition of the bilinear form, which
needs to use the transfer operator between the approximation space
and the test function space. We emphasize that, compared to the MsDFEM, the bilinear form of MsDPGM is not just choosing the discontinuous piecewise linear function space as test function space. More delicate choice of the terms of bilinear form should be made. The second one is the error estimate of MsDPGM.  We give the $H^2$ norm error estimate between the multiscale solution and the homogenized solution with the first order corrector, which plays an important role in the error estimate. The MsDPGM takes advantages of the MsPGM and DGM, which is expected to better approximate the multiscale solution than the standard MsPGM.
}

The proposed method is related with a combined finite element and oversampling multiscale Petrov--Galerkin method (FE-OMsPGM) \cite{SDW2015}. The idea of FE-OMsPGM is to utilize the traditional FEM directly on a fine mesh of the problematic part of the domain and use the OMsPGM on a coarse mesh of the other part. The transmission condition across the FE-OMsPGM interface is treated by the penalty technique of DGM. In \cite{SDW2015}, they deal with the transmission condition by penalizing the jumps from linear function values as well as the fluxes of the finite element solution on the fine mesh to those of the oversampling multiscale solution on the coarse mesh. Compared to \cite{SDW2015}, in this paper, we develop and analyze MsDPGM for the multiscale elliptic problems. The basic idea is to use PG formulation based on the discontinuous multiscale approximation space. The jump terms across each inter-element are dealt with penalty technique. The penalty term of linear function values is taken as that of the FE-OMsPGM, while the penalty term of the fluxes is not needed here.


Although the error analysis is given under the assumption
that the oscillating coefficient is periodic, our method is not
restrict to the periodic case. The numerical results show that the
introduced MsDPGM is very efficient for randomly generated
coefficients. Recently, the multiscale methods on localization of
the elliptic multiscale problems with highly varying (non-periodic)
coefficients are studied in some papers. For instance, the new
variational multiscale method is presented in \cite{Mp14}; a new oversampling
strategy for the MsFEM is presented in \cite{Hp13}. In the future
work, we will give more extensions and developments on our method
with the new oversampling strategy.

The outline of this paper is as follows. In Section 2, we give the
model problem and recall
the DG variational formulation of the model problem in the broken Sobolev spaces. Section 3 is devoted
to derive the MsDPGM. It includes the introduction of discontinuous oversampling multiscale approximation space and the derivation of
the formulations of { MsDFEM and MsDPGM}. In Section 4, we review the homogenization results and give some preliminaries for the error analysis. In Section 5, we give the main results of our method. It includes the stability and a priori error estimate of the proposed method. In Section 6, we first give several
numerical examples with periodic coefficients to demonstrate the
accuracy of the method. {Then we do the experiment to study how
the size of oversampling elements affects the errors.} Finally, we apply our method to multiscale
problems on the L--shaped domain to demonstrate the efficiency of the
method. Conclusions are given in the last section.

\section{Model problem and DG variational formulation}
In this section we introduce the multiscale model problem and give the DG variational formulation of the model problem.
First we state some notations and conventions. Throughout this paper, the Einstein summation convention is used: summation is taken over repeated indices.
Standard notation on Lebesgue and Sobolev spaces is employed.
Subsequently $C, C_0, C_1, C_2,\cdot \cdot \cdot$ denote generic constants,
which are independent of $ \ep, h $, unless otherwise stated.
We also use the shorthand notation $A \lesssim B$ and $B \lesssim A$ for the inequality $A \leq CB$ and $B \leq CA$. The notation $A \eqsim B$ is equivalent to the statement $A \lesssim B$ and $B \lesssim A$.
\subsection{Model problem}
Let $\Om\subset\mathbf{R}^n, n=2,3 $ be a bounded polyhedral domain.
Consider the following multiscale elliptic problem:
\begin{equation}\label{eproblem}
\left\{\begin{aligned} -\nabla\cdot(\mathbf{a}^\ep(x)\nabla
u_\ep(x)) & = f(x)&&
  \text{in}\,\Omega, \\
  u_\ep(x)&=0&&   \text{on}\,\partial\Omega,
\end{aligned}\right.
\end{equation}
where $\epsilon\ll 1$ is a parameter that represents the small scale
in the physical problem, $f \in L^2(\Om)$, and
$\mathbf{a}^\ep(x)=(a_{ij}^\ep(x))$ is a symmetric, positive
definite matrix:
\begin{equation}
  \lambda |\xi|^2\leq a_{ij}^\ep(x)\xi_i\xi_j\leq\Lambda |\xi|^2\quad
  \forall \xi\in \mathbf{R}^n,\, x\in\bar{\Omega}
\end{equation}
for some positive constants $\lambda$ and $\Lambda$.

\subsection{DG variational Formulation} In this subsection, we derive the DG variational formulation of the
model problem in the broken Sobolev spaces.
Let ${\cal T}_{h}$ be a quasi-uniform
triangulation of the domain $\Om$. We define $h_K$ as diam $(K)$ and
denote by $h=\text {max}_{K\in{\cal T}_h}h_K$.

We introduce the broken Sobolev spaces for any real number $s$,
$$ H^s({\cal T}_h) = \{v\in L^2(\Om):  \forall K \in {\cal T}_h, v|_K \in H^s(K)\}, $$
equipped with the broken Sobolev norm:
$$ |||v|||_{H^s({\cal T}_h)} = \Big(\sum_{K \in {\cal T}_h} ||v||^2_{H^s(K)}\Big)^{1/2}.$$

Denote by $\Ga_h$ the set of interior edges/faces of the ${\cal T}_h$.
With each edge/face $e$, we associate a unit normal vector $\bn$. If $e$ is on the boundary $\pa \Om$, then $\bn$ is taken to be the unit outward vector normal to $\pa \Om$.

If $v \in H^1(\T_h)$, the trace of $v$ along any side of one element $K$ is well defined.
If two elements $K_1^e$ and $K_2^e$ are neighbors and share one common side $e$, there are two traces of $v$ along $e$.
We define the average and jump for $v$. We assume that the normal vector $\bn$ is oriented from $K_1^e$ to $K_2^e$:
\begin{equation}\label{eja}
    \av{v}:=\frac{v|_{K_1^e}+v|_{K_2^e}}{2},\qquad \jm{v}:=v|_{K_1^e}-v|_{K_2^e} \qquad \forall e=\pa K_1^e \cap \pa K_2^e.
 \end{equation}
We extend the definition of jump and average to sides that belong to the boundary $\pa \Om$:
 \begin{equation*}
    \av{v} = \jm{v}=v|_{K_1^e} \qquad \forall e=\pa K_1^e \cap \pa \Om.
 \end{equation*}

In the following, we assume that $s=2$. Multiplying \eqref{eproblem} by any $v\in H^s({\cal T}_h)$, integrating on each element $K$, and using integration by parts, we obtain
\begin{equation*}
  \int_{K}\mathbf{a}^\ep\na u_\ep\cdot\na v\dx\nn
      -\int_{\pa K} \mathbf{a}^\ep\na u_\ep\cdot\bn_K v\ds=\int_{K}fv.
\end{equation*}
We recall that $\bn_K$ is the outward normal to $K$. Summing over all elements, and switching to the normal vectors $\bn$, we have
\begin{equation*}
     \sum_{K\in {\cal T}_h}\int_{\pa K} \mathbf{a}^\ep\na u_\ep\cdot\bn_K v\ds=\sum_{e\in\Ga_h\cup \pa \Om}\int_{e} \jm{\mathbf{a}^\ep\na u_\ep\cdot\bn v}\ds.
\end{equation*}
From the regularity of the solution $u_\ep$, it follows that
\begin{align}\label{wf}
  \sum_{K\in {\cal T}_h} \int_{K}\mathbf{a}^\ep\na u_\ep\cdot\na v\dx\nn
      -\sum_{e\in\Ga_h\cup \pa \Om}\int_{e} \av{\mathbf{a}^\ep\na u_\ep\cdot\bn}\jm{v}\ds=\int_{\Om}fv,
\end{align}
where we have used the formula $\jm{vw}=\av{v}\jm{w}+\jm{v}\av{w}$ and the fact that \\
{$\jm{\mathbf{a}^\ep\na u_{\ep}\cdot\bn}=0$}.

We now define the DG bilinear form $a(\cdot,\cdot): H^s({\cal T}_h)\times H^s({\cal T}_h)\rightarrow \mathbf{R}:$
\begin{equation*}
\begin{split}
  a(u,v):&=\sum_{K\in {\cal T}_h} \int_{K}\mathbf{a}^\ep\na u\cdot\na v\dx\nn
      -\sum_{e\in\Ga_h\cup \pa \Om}\int_{e} \av{\mathbf{a}^\ep\na u\cdot\bn}\jm{v}\ds\nn\\
      &+\be \sum_{e\in\Ga_h\cup \pa \Om}\int_{e} \jm{u}\av{\mathbf{a}^\ep\na v\cdot\bn}\ds\nn
      +\sum_{e\in\Ga_h\cup \pa \Om}\frac{\ga_0}{\rho}\int_{e} \jm{u}\jm{v}\ds,\nn
\end{split}
\end{equation*}
where $\beta$ is a real number such as $-1, 0, 1$, $\gamma_0$ is called penalty parameter, and $\rho>0$ will be specified later.

The general DG variational formulation of the problem \eqref{eproblem} is as follows:
Find $u_\ep \in H^s({\cal T}_h)$, such that
\begin{equation}\label{weakformula}
      a(u_\ep,v)=(f,v) \qquad\forall\, v\in H^s({\cal T}_h).
\end{equation}
 \begin{remark}
It is easy to check that if the solution $u_\ep$ of problem \eqref{eproblem} belongs to $H^2(\Om)$, then $u_\ep$ satisfies the variational formulation \eqref{weakformula}.
Conversely, if $u_\ep \in H^1(\Om)\cap H^s({\cal T}_h)$ satisfies \eqref{weakformula}, then $u_\ep$ is the solution of problem \eqref{eproblem}.
\end{remark}

\section{Multiscale discontinuous Petrov-Galerkin method}
{This section is devoted to the formulations of multiscale discontinuous methods for solving \eqref{eproblem}. In subsection 3.1, we introduce the oversampling multiscale approximation space defined on the triangulation ${\cal T}_h$. The formulations of the MsDFE and MsDPG methods are presented in subsection 3.2.}

\subsection{Oversampling multiscale approximation space}
In this subsection, we introduce the oversampling multiscale approximation space defined on the triangulation ${\cal T}_h$ (cf. \cite{CW2010,EH2009,HW}). Here we only consider the case where $n=2$. For any $K\in {\cal T}_h$ with nodes
$\{x_i^K\}_{i=1}^3$, let $\{\varphi_i^K\}^3_{i=1}$ be the basis of
$P_1(K)$ satisfying $\varphi_i^K(x_j^K)=\delta_{ij}$, where
$\delta_{ij}$ stands for the Kroneckers symbol. For any $K\in
{\cal T}_h,$ we denote by $S=S(K)$ a macro-element which contains
$K$  {and $d_{K} = \text{dist}(\partial
S,K)$. We assume that
 $d_{K}\geq \delta_0 h_K$ for some
positive constant $\delta_0$ independent of $h_K$. The minimum
angle of $S(K)$ is bounded from below by some positive constant
$\theta_0$ independent of $h_K$. In our later numerical experiments,
for any coarse-grid element $K\in\T_{h}$ we put its macro-element
$S(K)$ in such a way that their barycenters are coincide and
their corresponding edges are parallel.
See Figure~\ref{figsample1} for an illustration.}

%
%
%
%
%
%
%

Let $\psi_i^S,i=1,2,3,$ with $\psi_i^S\in H^1(S)$, be the solution of
the problem:
\begin{equation}\label{overbase}
  -\nabla\cdot (\mathbf{a}^\ep\nabla\psi_i^S)=0\qquad \text{ in } S,\qquad \psi_i^S|_{\partial S}=\varphi_i^S.
\end{equation}
Here $\{\varphi_i^S\}_{i=1}^3$ is the nodal basis of $P_1(S)$ such
that $\varphi_i^S(x_j^S)=\delta_{ij},i,j=1,2,3.$

The oversampling
multiscale basis functions on $K$ are defined by
{\begin{equation}\label{coefcij1}
 \bar{ \psi_i}^K=\sum_{j=1}^{3}c_{ij}^K\psi_j^S|_K\qquad \text{ in } K,
\end{equation}
with the constants so chosen that
\begin{equation}\label{coefcij2}
  \varphi_i^K=\sum_{j=1}^{3}c_{ij}^K\varphi_j^S|_K\qquad \text{ in } K.
\end{equation}}
The existence of the constants $c_{ij}^K$ is guaranteed because
$\{\varphi_j^S\}_{j=1}^3$ also forms the basis of $P_1(K)$.
{To illustrate the basis functions, we depict two examples of them in Figure~\ref{fig:base} (cf. \cite{SDW2015}) .
\begin{figure}[htp]
  \begin{minipage}[t]{0.5\linewidth}
  \centerline{\includegraphics[scale=0.58]{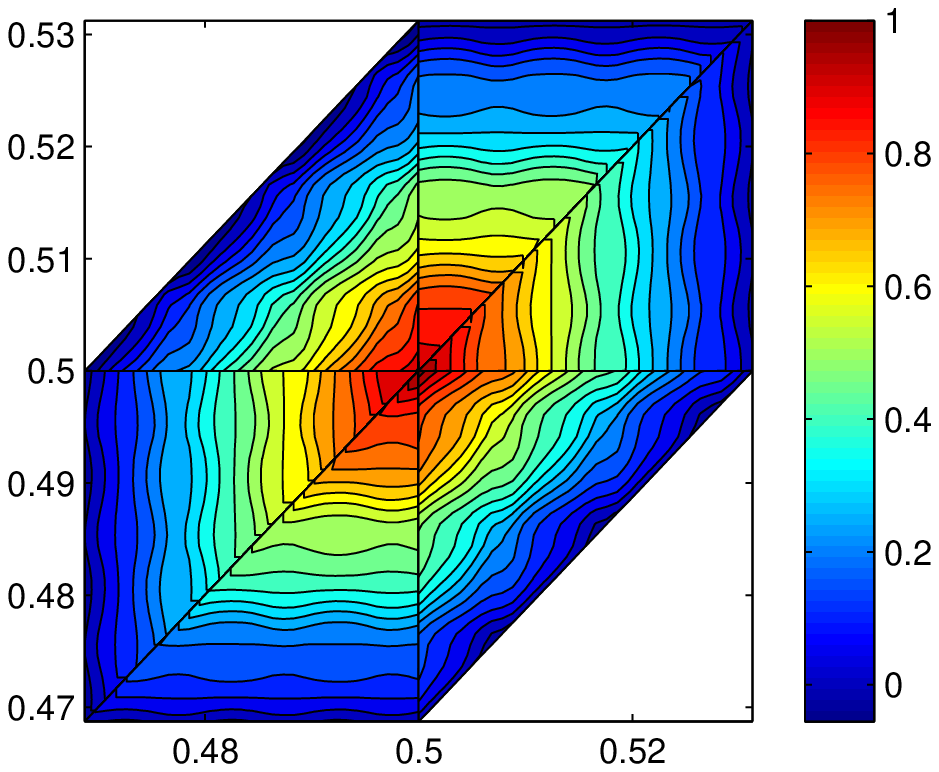}}
  \end{minipage}
  \begin{minipage}[t]{0.5\linewidth}
  \centerline{\includegraphics[scale=0.58]{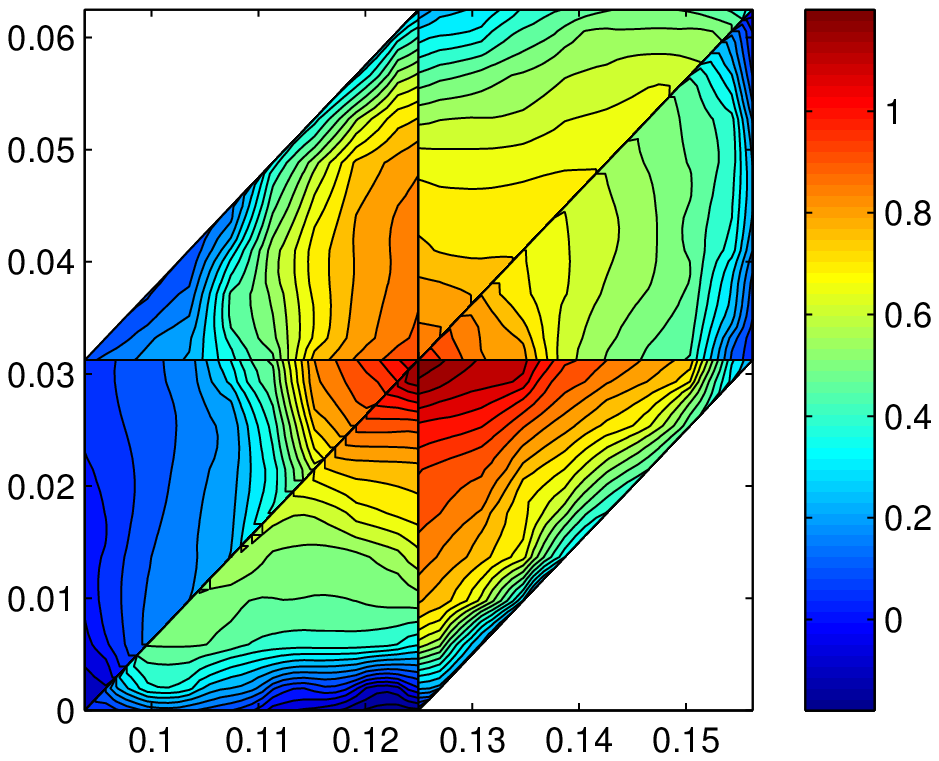}}
  \end{minipage}
\caption{Example of oversampling basis functions. Left: basis function for periodic media. Right: basis function for random media.  }
\label{fig:base}
\end{figure}}

Let $\text{OMS}(K)=\mathrm{span\,}\{\bar{\psi_i}^K\}_{i=1}^3$ { be the set of space functions on $K$.}
Define
the projection $\Pi_K:\text{OMS}(K)\rightarrow P_1(K)$ as follows:
\begin{equation*}
  \Pi_K\psi=c_i\varphi_i^K\qquad \text{ if }\qquad \psi=c_i\bar\psi_i^K\in \text{OMS}(K).
\end{equation*}


{Further, we introduce the discontinuous piecewise ``OMS'' approximation space and the discontinuous piecewise linear space:}
{\begin{equation*}
\begin{split}
  V_{h,dc}^{ms}&=\{\psi_h\in L^2(\Om):\psi_h|_K\in\text{OMS}(K)\qquad \forall K\in {\cal T}_h \},\\
  V_{h,dc}&=\{v_h\in L^2(\Om): v_h|_K\in P_1(K)\qquad \forall K\in {\cal T}_h \}.
  \end{split}
\end{equation*}}
Here we use the abbreviated indexes `ms', `dc' for multiscale, discontinuous, respectively.

{\subsection{Formulation of the MsDFEM/MsDPGM}
In this subsection we present the formulations of the MsDFEM and MsDPGM.

 By use of the DG variational formulation~(\ref{weakformula}) and the discontinuous piecewise ``OMS'' approximation space, we are now ready to define
the MsDFE method:
Find $\widetilde{u}_{h}\in V_{h,dc}^{ms}$ such that
\begin{equation}\label{eifem}
a(\widetilde{u}_{h}, v_{h}) =(f, v_{h}) \qquad\forall\, v_{h}\in V_{h,dc}^{ms}.
\end{equation}

}

To define the discrete bilinear form for MsDPGM, we need the transfer operator $\Pi_h: V_{h,dc}^{ms} \rightarrow V_{h,dc}$ as following:
\begin{equation*}
    \Pi_h\psi_h|_K=\Pi_K\psi_h\ \ \mbox{for any $K\in{\cal T}_h , \psi_h\in V_{h,dc}^{ms}$}.
\end{equation*}

\begin{remark}
{ In general, the trial and test functions of PGM are not in the same space. For example, here we might use $V_{h,dc}^{ms}$ and $V_{h,dc}$ as the trial function and test function spaces respectively. However, it may result in a difficulty to prove the inf-sup condition of the corresponding bilinear form. Hence, in this paper we introduce the transfer operator $\Pi_h$ to connect the two spaces, and use it in the bilinear form which causes an easy way to establish the stability of the MsDPGM. The idea of connecting the trial function and test function spaces in the Petrov-Galerkin method through an operator was introduced in \cite{Chou1997Analysis} (see also \cite{Morton1980Petrov}). }
\end{remark}

The discrete bilinear form of MsDPGM on $V_{h,dc}^{ms}\times V_{h,dc}^{ms}$ is defined as:
\begin{align}\label{eA}
a_{h}(u_{h},v_{h}) :=&\sum_{K\in {\cal T}_{h}} \int_{K}\mathbf{a}^\ep\na u_{h}\cdot\na \Pi_h v_{h}\dx\nn\\
&-\sum_{e\in\Ga_h\cup \pa \Om}\int_{e} \av{\mathbf{a}^\ep\na u_{h}\cdot\bn}\jm{\Pi_h v_{h}}\ds\nn\\
&+\be \sum_{e\in\Ga_h\cup \pa \Om}\int_{e} \jm{\Pi_h u_{h}}\av{\mathbf{a}^\ep\na v_{h}\cdot\bn}\ds\nn\\
&+  J_0(u_{h},v_{h}) ,\nn \\
J_0(u_{h},v_{h}):=&\sum_{e\in\Ga_h\cup \pa \Om}\frac{\ga_0}{\rho}\int_{e} \jm{\Pi_h u_{h}}\jm{\Pi_h v_{h}}\ds,\nn
\end{align}
where $\beta$ is a real number such as $-1, 0, 1$, { $\gamma_0$ is the penalty parameter, and $\rho$ will be specified later.}

{\begin{remark}
It is well known that DG methods utilize discontinuous piecewise
polynomial functions and numerical fluxes, which implies that the
weak formulation subject to discretization must include jump terms
across interfaces and that some penalty terms must be added to
control the jump terms. Therefore, the methods need the restriction on the penalty parameter to ensure stability and convergence in some sense. In fact, the optional convergence is related with the penalty parameter (see \cite{Epshteyn2007Estimation}).
In this paper, our MsDPGM takes advantage of the penalty technique, which is inevitable to consider the choice of the penalty parameter. In our theoretical analysis, the penalty parameter $\gamma_0$ is constrained by a large constant from below to ensure the coercivity of $a_h$. However, in practice, the penalty parameter is chosen through our experience. In later numerical tests, we try different choice of the penalty parameter to study its affection on the error.
\end{remark}
\begin{remark}
The parameter $\rho$ is chosen as $\epsilon$ in our later error analysis, while in practical computation, it may be chosen as the mesh size $h$.
\end{remark}
}

Then, our MsDPG method is:
 Find $u_{h}\in V_{h,dc}^{ms}$ such that
\begin{equation}\label{eipgfem}
a_h(u_{h}, v_{h}) =(f, \Pi_hv_{h}) \qquad\forall\, v_{h}\in V_{h,dc}^{ms}.
\end{equation}

\begin{remark}
The design of the last two terms in $a_h$ is tricky. As a matter
of fact, we have tried numerically different possibilities of using
$\Pi_h$ ( or not before each $u_{h}$ or $v_{h}$ ) before we found
that the current form of $a_h$ is the best one and, most
importantly, the corresponding MsDPGM can be analyzed theoretically. {Indeed, our MsDPGM is some kind of pseudo Petrov-Galerkin formulation of the method that the test function space is
formally the same as the solution space, however some terms
involve a projection of the multiscale test function into a piecewise linear function space. }
\end{remark}

We denote the discrete norm for MsDPGM on $V_{h,dc}^{ms}$,
 \begin{equation*}
 \begin{split}
 \norm{v_h}_{h,\Om}:=&\Big(\sum_{K\in\T_h}\int_K \mathbf{a}^\ep \na v_h \cdot \na v_h \dx+\sum_{e\in\Ga_h\cup \pa \Om}\frac{\rho}{\ga_0}\int_{e} \av{\mathbf{a}^\ep \na v_h\cdot\bn}^2\ds\\
  &\qquad\qquad+\sum_{e\in\Ga_h\cup \pa \Om}\frac{\ga_0}{\rho}\int_{e} \jm{\Pi_hv_h}^2\ds
 \Big)^{1/2}.
 \end{split}
 \end{equation*}

Noting that the operator $\Pi_h$ is not defined for the exact solution $u_\ep$,  we introduce the following function to measure the error of the discrete solution:
\begin{equation}\label{err}
    \begin{split}
      E(v,v_{h}):=&\bigg(\sum_{K\in\T_{h}}\|(\mathbf{a}^\ep)^{1/2}\na (v-v_{h})\|_{L^2(K)}^2\\
&+\sum_{e\in\Ga_h\cup \pa \Om} \frac{\rho}{\ga_0}\norml{\av{\mathbf{a}^\ep\na (v-v_{h})\cdot\bn}}{e}^2\\
&+\sum_{e\in\Ga_h\cup \pa \Om} \frac{\ga_0}{\rho}\norml{\jm{v-\Pi_h v_{h}}}{e}^2\bigg)^{1/2}\quad\forall v\in H^2(\Om), v_{h}\in V_{h,dc}^{ms}.
    \end{split}
\end{equation}

From the triangle inequality, it is clear that, for any  $v\in H^2(\Om), v_{h}, w_{h}\in V_{h,dc}^{ms}$,
\begin{align}\label{errtri}
E(v,v_{h})\ls E(v,w_{h})+\norm{w_{h}-v_{h}}_{h,\Om}.
\end{align}

{\section{Homogenization results and preliminaries}
In this section we first review the results of classical homogenization theory, and give an important $H^2$ norm error estimate between the multiscale solution and the homogenized solution with the first order corrector. Then we recall some preliminaries for our following analysis.
\subsection{The homogenization results}
Hereafter, we assume that $\mathbf{a}^\varepsilon(x)$ has the
form $\mathbf{a}(x/\varepsilon)$ and
$a_{ij}(y)$ are sufficiently smooth periodic functions in $y$ with respect to a
unit cube $Y$. For our analysis it is sufficient to assume that
$a_{ij}(y)\in W^{2,p}(Y)$ with $p>n$.




For convenience sake, we take
\begin{equation*}
    u_1=u_0+\ep\chi^j(x/\varepsilon) \frac{\pa u_0}{\pa x_j},
\end{equation*}
where $u_0$ is the homogenized solution, $\chi^j$ is the periodic solution of the following cell problem (cf. \cite{BLP,JKO}):
\begin{equation}
-\nabla_y\cdot(\mathbf{a}(y)\nabla_y
\chi^j(y))=\nabla_y\cdot(\mathbf{a}(y)\mathbf{e}_j),\quad
j=1,\cdots,n
\end{equation}
with zero mean, i.e., $\int_Y\chi^jdy=0$, and $\mathbf{e}_j$ is the
unit vector in the $j$th direction.

The following theorem gives the $H^2$ semi-norm estimate of the error $u_\ep-u_1$, which plays a key role in the later error analysis.{We arrange the proof in the Appendix~\ref{AA}.}


\begin{Theorem}\label{h2homoerror}
Assume that $u_0\in H^3(\Om)$. Then the following estimate is valid:
\begin{equation}\label{h2homoestimate}
\abs{u_\ep-u_1}_{H^2(\Om)}\ls
\abs{u_0}_{H^2(\Om)}+\frac{1}{\sqrt{\ep}}\abs{u_0}_{W^{1,\infty}(\Om)}+\ep\abs{u_0}_{H^3(\Om)}.
\end{equation}
\end{Theorem}

\subsection{Preliminaries}
In this subsection, we give some preliminaries for our later analysis. We first recall the definition of $\psi_i^S, i=1,2,3$ (see \eqref{overbase}).
By the asymptotic expansion (cf. \cite{EHW,HWZh}), we know that
\begin{equation}\label{asy}
    \psi_i^S=\varphi _i^S+\ep \chi^j(x/\varepsilon) \frac{\pa \varphi_i^S}{\pa x_j}+\ep \eta^j(x) \frac{\pa \varphi _i^S}{\pa x_j},
\end{equation}
with $\eta^j$ being the solution of
\begin{equation}\label{etadef}
    -\na \cdot(\mathbf{a}^\ep\na\eta^j) =0  \qquad \text{ in } S,\qquad \eta^j|_{\partial S}=-\chi^j(x/\varepsilon).
\end{equation}
Substituting \eqref{asy} to \eqref{coefcij1}, we see that $\bar{
\psi_i}^K$ can be expanded as follows:
\begin{equation}\label{over}
\bar{ \psi_i}^K=\varphi _i^K+\ep \chi^j(x/\varepsilon) \frac{\pa \varphi _i^K}{\pa x_j}+\ep \eta^j(x) \frac{\pa \varphi _i^K}{\pa x_j}.
\end{equation}
{Recall that $d_{K} = \text{dist}(\partial
S,K)$, which satisfies:
$ d_K\geq \delta_0 h_K$.}  Denote by $d=\text{min}_{K\in \T_h}d_{K}$.
{\begin{remark}
It has been shown in \cite{HW,HWZh} that the distance $d_K$ is determined by the thickness of the boundary layer of $\eta^j$. Numerically, it has been observed that the boundary layer of $\eta^j$ is about $O(\epsilon)$ thick (see \cite{HW}). It was also observed that $d_K = h_K (> \epsilon)$ is usually sufficient for eliminating the boundary layer effect. Therefore in our numerical tests we choose $h_K$ as the oversampling size in this paper. To study how the size of oversampling elements affects the errors, in Section~\ref{numericaltest} we include a numerical test which use a series of  $d_K$ with different $\delta_0$ to compare the corresponding errors.
\end{remark}
}

By the Maximum Principle we have
\begin{equation}\label{eq1}
\norm{\eta^{j}}_{L^\infty(S)}\leq \abs{\chi^{j}}_{L^\infty(S)}\lesssim 1,
\end{equation}
which together with the interior gradient estimate (see \cite[Lemma 3.6]{DW2014} or \cite[Proposition C.1]{EHW}) imply that
\begin{equation}\label{eq2}
\norm{\na \eta^{j}}_{L^\infty(K)} \lesssim \frac{1}{d_{K}}.
\end{equation}

Next, we give a trace inequality which will be used in this paper frequently
(see \cite[Theorem 1.6.6]{SUSC}, \cite{Ciarlet}).
\begin{lemma}\label{trace1}Let $K$ be an element of the triangulation ${\cal T}_h$. Then, for any $v\in H^1(K)$, we have
\begin{equation}\label{trace}
    \norml{v}{\partial K}\leq C\left(\emph{diam}(K)^{-1/2}\norml{v}{K}
    +\norml{v}{K}^{1/2}\norml{\na v}{K}^{1/2}\right).
\end{equation}
\end{lemma}
The following lemma gives some approximation properties of the space
OMS(K) (cf. \cite[Lemma 4.1]{DW2014}).
\begin{lemma} \label{lmultiapp}Take $\phi_{h}^K=\sum\limits_{x_i^K\, \text{node of}\ K}u_0(x_i^K)\bar\psi_i^K(x),\quad  \forall K\in {\cal T}_h$. Then, the following estimates hold:
\begin{align}
\abs{u_1-\phi_{h}^K}_{H^1(K)}&\ls h_K \abs{u_0}_{H^2(K)}
   +{\eps}h_{K}^{n/2}d_{K}^{-1}\abs{u_0}_{W^{1,\infty}(K)},\label{lmultiappl1}\\
   \norm{u_1-\phi_{h}^K}_{L^2(K)}&\ls h_K^2\abs{u_0}_{H^2(K)}
   +{\eps}h_{K}^{n/2}\abs{u_0}_{W^{1,\infty}(K)},\label{lmultiappl2}\\
    \abs{u_1-\phi_{h}^K}_{H^2(K)}&\ls \ep^{-1}h_K\abs{u_0}_{H^2(K)}+h_{K}^{n/2}d_{K}^{-1}\abs{u_0}_{W^{1,\infty}(K)}
   +\ep\abs{u_0}_{H^3(K)}.\label{lmultiappl3}
\end{align}
\end{lemma}

Moreover, we recall the stability estimate for $\Pi_K$, which will be
used in our later analysis (cf. Lemma 3.2 in \cite{SDW2015}).
\begin{lemma}\label{stab1}
  There exist positive constants $\ga$, $\alpha_1$ and $\alpha_2$ which are independent of $h$ and $\eps$ such that
 if $\eps/h_K\leq\ga $ for all $K\in\T_h$, then the following estimates are valid for any  $v_h\in \text{OMS}(K)$,
\begin{eqnarray}
 &&\norm{\na v_h}_{L^2(K)}
   \eqsim \norm{\na \Pi_K v_h}_{L^2(K)},\label{stab2}\\
 && \alpha_2\norml{\na v_{h}}{K}^2 \le \abs{\int_{K}\mathbf{a}^\ep\na v_{h}\cdot\na
    \Pi_K v_h\dx}\le \alpha_1\norml{\na v_{h}}{K}^2. \label{stab3}
\end{eqnarray}
\end{lemma}
The following lemma gives an inverse estimate for the function in
space $\text{OMS}(K)$ \cite[Lemma 5.2]{DW2014}.
\begin{lemma}\label{inverse1}
Under the assumptions of Lemma \ref{stab1}, and assuming $\ep\lesssim h_K\lesssim
d_K$, we have
\begin{equation}
    |v_{h}|_{H^{2}(K)}\lesssim \frac{1}{\ep}\norml{\na v_{h}}{K}  \quad\forall\,
v_{h}\in OMS(K).
\end{equation}
\end{lemma}
\section{Main results}
{In this section we only carry out the convergence analysis of MsDPGM. For the case of MsDFEM, similar results can be obtained by the same argument and are arranged in the Appendix~\ref{AB} for convenience of the reader.} For MsDPGM, we first show the stability of the bilinear form guaranteeing the existence and uniqueness of the solution, and then prove the error estimate with $\beta=-1, \rho=\ep$. For other cases such as $\beta=0,1$, the analysis is similar and is omitted here.
\subsection{Existence and uniqueness of the solution of the MsDPGM}
 We start by establishing the stability of the bilinear form of the MsDPGM.
\begin{Theorem}\label{ccPG}We have
\begin{equation}\label{cont}
    \abs{a_h(u_{h},v_{h})}\leq C\norm{u_{h}}_{h,\Om}\norm{v_{h}}_{h,\Om}  \qquad\forall\, u_{h},v_{h} \in V_{h,dc}^{ms}.
\end{equation}
Further, let the assumptions of Lemma \ref{inverse1} be fulfilled and $\ga_0$ is large enough, then
\begin{equation}\label{coer}
   a_h(v_{h}, v_{h})\geq \kappa\norm{v_{h}}_{h,\Om}^2  \qquad\forall\, v_{h} \in V_{h,dc}^{ms},
\end{equation}
where $\kappa>0$ is a constant independent of $h, \ep, \ga_0$.
\end{Theorem}
\begin{proof}
{\color{black} From the definition of the norms, the Cauchy-Schwarz inequality and  Lemma~\ref{stab1}, it follows \eqref{cont} immediately.}

Next we prove \eqref{coer}. From \eqref{stab3}, we get
 \begin{equation*}
    \begin{split}
a_h(v_{h}, v_{h})&\geq C\sum_{K\in {\cal T}_{h}} \norml{(\mathbf{a}^\ep)^{1/2}\na v_{h}}{K}^2
- 2\sum_{e\in\Ga_h\cup \pa \Om}\int_{e} \av{\mathbf{a}^\ep\na v_{h}\cdot\bn}\jm{\Pi_h v_{h}}\ds\\
&\qquad+\sum_{e\in\Ga_h\cup \pa \Om}\frac{\ga_0}{\ep}\norml{\jm{\Pi_h v_{h}}}{e}^2.
     \end{split}
\end{equation*}
It is easy to see that,
\begin{align*}
   &2\sum_{e\in\Ga_h\cup \pa \Om}\int_{e} \av{\mathbf{a}^\ep\na v_{h}\cdot\bn}\jm{\Pi_h
   v_{h}}\ds
   \leq 2\sum_{e\in\Ga_h\cup \pa \Om}\norml{\av{\mathbf{a}^\ep\na v_{h}\cdot
\bn}}{e}\norml{\jm{\Pi_h v_{h}}}{e}\\
&\qquad\leq \sum_{e\in\Ga_h\cup \pa \Om}\frac{\ga_0}{2 \ep}\norml{\jm{\Pi_h v_{h}}}{e}^2
+\sum_{e\in\Ga_h\cup \pa \Om}\frac{2 \ep}{\ga_0}\norml{\av{\mathbf{a}^\ep\na v_{h}\cdot
\bn}}{e}^2.
     \end{align*}
Hence, we obtain
{ \begin{equation}\label{eqconce1}
    \begin{split}
a_h(v_{h}, v_{h})&\geq C\sum_{K\in {\cal T}_{h}} \norml{(\mathbf{a}^\ep)^{1/2}\na v_{h}}{K}^2-\frac{1}{2}\sum_{e\in\Ga_h\cup \pa \Om}\frac{\ga_0}{\ep}\norml{\jm{\Pi_h v_{h}}}{e}^2\\
&\quad -2\sum_{e\in\Ga_h\cup \pa \Om}\frac{\ep}{\ga_0}\norml{\av{\mathbf{a}^\ep\na v_{h}\cdot
\bn}}{e}^2+\sum_{e\in\Ga_h\cup \pa \Om}\frac{\ga_0}{\ep}\norml{\jm{\Pi_h v_{h}}}{e}^2\\
&= C\sum_{K\in {\cal T}_{h}} \norml{(\mathbf{a}^\ep)^{1/2}\na v_{h}}{K}^2+\frac{1}{2}\sum_{e\in\Ga_h\cup \pa \Om}\frac{\ga_0}{\ep}\norml{\jm{\Pi_h v_{h}}}{e}^2\\
&\quad +\frac{1}{2}\sum_{e\in\Ga_h\cup \pa \Om}\frac{\ep}{\ga_0}\norml{\av{\mathbf{a}^\ep\na v_{h}\cdot
\bn}}{e}^2- \frac{5}{2}\sum_{e\in\Ga_h\cup \pa \Om}\frac{\ep}{\ga_0}\norml{\av{\mathbf{a}^\ep\na v_{h}\cdot
\bn}}{e}^2.
     \end{split}
\end{equation}
}
By use of Lemmas \ref{trace1}, \ref{inverse1} and $\ep\ls h$, we have
\begin{equation}\label{eqconce2}
    \frac{\ep}{\ga_{0}}\norml{\{\mathbf{a}^\ep\na
v_{h}\cdot \bn\}}{e}^{2} \leq
\frac{C_{1}}{\ga_{0}}\norml{(\mathbf{a}^\ep)^{1/2}\na
v_{h}}{K}^{2}.
\end{equation}
Therefore, from \eqref{eqconce1} and \eqref{eqconce2}, we have
\begin{equation*}
    \begin{split}
a_h(v_{h}, v_{h})&\geq \left(C-\frac{5C_1}{2\ga_0}\right)\sum_{K\in {\cal T}_{h}} \norml{(\mathbf{a}^\ep)^{1/2}\na v_{h}}{K}^2\\
&\quad +\frac{1}{2}\sum_{e\in\Ga_h\cup \pa \Om}\frac{\ep}{\ga_0}\norml{\av{\mathbf{a}^\ep\na v_{h}\cdot
\bn}}{e}^2\\
&\quad +\frac{1}{2}\sum_{e\in\Ga_h\cup \pa \Om}\frac{\ga_0}{\ep}\norml{\jm{\Pi_h v_{h}}}{e}^2,
\end{split}
\end{equation*}
where $\ga_0$ is large enough such that
$\frac{5C_1}{2\ga_0}<\frac{C}{2}$. Then by choosing $\kappa=\text {min}(
\frac{C}2,\frac12)$, it follows \eqref{coer}. This completes the
proof.
\end{proof}

Theorem \ref{ccPG} guarantees that there exists a unique solution to our MsDPGM. Now we establish an analogue of C\'ea lemma written in the following theorem:
\begin{Theorem}\label{lcea} For large enough $\ga_0$, the following inequality holds:
\begin{equation}
    E(u_\ep,u_{h})\lesssim \inf_{v_{h}\in V_{h,dc}^{ms}}E(u_\ep,v_{h}),
\end{equation}
where the error function $E$ is defined in \eqref{err}.
\end{Theorem}
\begin{proof}
It is clear that by Theorem \ref{ccPG} we have
\begin{equation*}
    \begin{split}
       \norm{u_{h}-v_{h}}^2_{h,\Om}\ls& a_h(u_{h}-v_{h}, u_{h}-v_{h})\\
      =&a_h(u_{h}, u_{h}-v_{h})-a_h(v_{h}, u_{h}-v_{h})\\
      =&(f,\Pi_h (u_{h}- v_{h}))-a_h(v_{h}, u_{h}-v_{h}).
    \end{split}
\end{equation*}
From \eqref{weakformula}, it follows that
\begin{equation*}
    \begin{split}
      (f,\Pi_h (u_{h}-v_{h}))=& \sum_{K\in {\cal T}_{h}}\int_{K}\mathbf{a}^\ep\na u_{\ep}\cdot \na \Pi_h(u_{h}-v_{h})\dx\\
      &\quad-\sum_{e\in\Ga_h\cup \pa \Om}\int_{e} \av{\mathbf{a}^\ep\na u_{\ep}\cdot\bn}\jm{\Pi_h (u_{h}-v_{h})}\ds.
     \end{split}
\end{equation*}
Then, using the facts that $\jm{u_\ep}=0$ and $\jm{\mathbf{a}^\ep\na u_{\ep}\cdot\bn}=0 $, we have
\begin{equation*}
    \begin{split}
      (f,\Pi_h (u_{h}- v_{h}))&-a_h(v_{h}, u_{h}-v_{h}) \\
      &=\sum_{K\in {\cal T}_{h}}\int_{K}\mathbf{a}^\ep\na (u_{\ep}-v_{h})\cdot \na \Pi_h (u_{h}-v_{h})\dx\\
      &\quad-\sum_{e\in\Ga_h\cup \pa \Om}\int_{e} \av{\mathbf{a}^\ep\na (u_{\ep}-v_{h})\cdot\bn}\jm{\Pi_h (u_{h}-v_{h})}\ds\\
      &\quad+\sum_{e\in\Ga_h\cup \pa \Om}\int_{e} \av{\mathbf{a}^\ep\na (u_{h}-v_{h})\cdot\bn}\jm{\Pi_h v_{h}-u_\ep}\ds\\
      &\quad-\sum_{e\in\Ga_h\cup \pa \Om}\frac{\ga_0}{\ep}\int_{e}\jm{\Pi_h v_{h}-u_\ep}\jm{\Pi_h (u_{h}-v_{h})}\ds\\
      &\ls E(u_{\ep},v_{h})\norm{u_{h}-v_{h}}_{h,\Om}.
     \end{split}
\end{equation*}
Therefore, we obtain
\begin{equation*}
      \norm{u_{h}-v_{h}}_{h,\Om} \ls E(u_{\ep},v_{h}),
\end{equation*}
which together with \eqref{errtri} yield
\begin{align*}
E(u_{\ep},u_{h})\ls E(u_{\ep},v_{h})+\norm{u_{h}-v_{h}}_{h,\Om}\ls E(u_{\ep},v_{h}).
\end{align*}
The proof is completed.
\end{proof}

\subsection{A priori error estimate of MsDPGM}
We present the main result of the paper which gives the error
estimate of the MsDPGM.
{\begin{Theorem}\label{energeerror}Let $u_\ep$ be the solution of \eqref{eproblem}, and let $u_h$ be the numerical solution computed using MsDPGM defined by \eqref{eipgfem}. Assume that $u_0 \in H^3(\Om), f \in L^2(\Om)$, and that $\ep \ls h \ls d$,
and that the penalty parameter $\ga_0$ is large enough. Then there
exits a constant $\ga$ independent of $h$ and $\ep$ such that if
$\ep/h_K \leq \ga$ for all $K \in \T_h$, the following error
estimate holds:
\begin{equation}\label{energyest}
  E(u_\ep,u_{h})\lesssim \sqrt{\ep}+\frac{\ep}{d}+h+\frac{h^{3/2}}{\sqrt{\ep}},
\end{equation}
where $d=\text {min}_{K\in \T_h}d_K$.
\end{Theorem}}

\begin{proof}
According to Theorem~\ref{lcea}, the proof is devoted to estimating
the interpolation error. To do this, we define $\psi_h$ by
   \begin{equation}\label{intplMsFEM}
        \psi_h|_K=\phi_h^K=\sum_{x_i^K\, \text{node of}\ K}u_0(x_i^K)\bar\psi_i^K(x)\quad  \forall K\in {\cal T}_h.
   \end{equation}
Clearly, $\psi_h\in V_{h,dc}^{ms}$. It is easy to see that
   \begin{equation*}
      \Pi_K\phi_h^K=I_hu_0|_{K},
   \end{equation*}
where $I_h:H^{s}(\T_h)\rightarrow V_{h,dc}$ is the Lagrange interpolation operator. Then we set $v_{h}$ as $\psi_h$. It is shown that in \cite{DW2014},
\begin{equation}\label{estr0}
\begin{split}
&\bigg(\sum_{K\in\T_{h}}\|(\mathbf{a}^\ep)^{1/2}\na (u_\ep-v_{h})\|_{L^2(K)}^2\bigg)^{1/2}\\
 &\qquad\qquad\ls  h\abs{u_0}_{H^2(\Om)}+\sqrt{\ep}|u_0|_{W^{1,\infty}(\Om)}+\frac{\ep}{d}|u_0|_{W^{1,\infty}(\Om)}.
 \end{split}
\end{equation}

Next, we estimate the term
\begin{equation*}
    \sum_{e\in\Ga_h\cup \pa \Om}\frac{\ep}{\ga_0}\norml{\av{\mathbf{a}^\ep\na (u_\ep-v_{h})\cdot\bn}}{e}^2:=\mathrm{I}.
\end{equation*}
From \eqref{trace}, we have
\begin{equation*}
\begin{split}
  \mathrm{I} &\ls \ep h^{-1}\norml{\na(u_\ep-u_{1})}{\Om}^2+\ep h^{-1}\sum_{K \in \T_h}\norml{\na(u_1-\psi_h)}{K}^2\\
   &\quad +\ep \norml{\na(u_{\ep}-u_1)}{\Om}^2\norml{\na^2(u_{\ep}-u_1)}{\Om}^2\\
   &\quad +\ep \Big(\sum_{K \in \T_h}\norml{\na(u_1-\psi_h)}{K}^2\Big)^{1/2}\Big(\sum_{K \in \T_h}\norml{\na^2(u_1-\psi_h)}{K}^2\Big)^{1/2}.
\end{split}
\end{equation*}
Therefore, it follows from Theorem \ref{h2homoerror}, Lemma \ref{lmultiapp} and the assumption $\ep \ls h \ls d$ that,
\begin{equation}\label{I}
    \mathrm{I}\ls h^2 |u_0|_{H^2(\Om)}^2+\ep |u_0|_{W^{1,\infty}(\Om)}^2+\ep^4 |u_0|_{H^3(\Om)}^2,
\end{equation}
where we have used $\frac{\ep}{\sqrt{h}} < \sqrt{\ep}$ and the Young's inequality to derive the above inequality.

It remains to consider the term $\sum_{e\in\Ga_h\cup \pa \Om} \frac{\ga_0}{\ep}\norml{\jm{u_\ep-\Pi_h v_{h}}}{e}^2.$ Noting that both $u_\ep$ and $u_0$ are continuous functions, we have
\begin{equation*}
\begin{split}
  \sum_{e\in\Ga_h\cup \pa \Om} \frac{\ga_0}{\ep}\norml{\jm{u_\ep-\Pi_h v_{h}}}{e}^2&=\sum_{e\in\Ga_h\cup \pa \Om} \frac{\ga_0}{\ep}\int_e\jm{u_0-\Pi_h v_{h}}^2\ds\\
  &\ls \sum_{e\in\Ga_h\cup \pa \Om}\frac{\ga_0}{\ep}\int_e(u_0-\Pi_h \psi_h)^2\ds:=\mathrm{II}.
\end{split}
\end{equation*}
Then, by use of
Lemma~\ref{trace1}, we have
\begin{equation*}
    \begin{split}
       \int_e(u_0-&\Pi_h \psi_h)^2\ds=\int_e(u_0-I_hu_0)^2\ds\\
       &\ls h^{-1}\norml{u_0-I_hu_0}{K}^2+\norml{u_0-I_hu_0}{K}\norml{\na(u_0-I_hu_0)}{K}\\
       &\ls h^3|u_0|_{H^{2}(K)}^2,
     \end{split}
\end{equation*}
which yields
\begin{equation}\label{estr1}
\mathrm{II}\ls\frac{h^3}{\eps}|u_0|_{H^2(\Om)}^2.
\end{equation}
Hence, from \eqref{estr0}--\eqref{estr1}, it follows \eqref{energyest} immediately.
\end{proof}

\section{Numerical experiments}\label{numericaltest}
{In this section, we present numerical experiments to confirm the theoretical results in Section 5. We show the numerical results of MsDPGM defined in (\ref{eipgfem}), and also results of MsDFEM defined in (\ref{eifem}) which show good performance as well as MsDPGM.
In order to illustrate the accuracy of our methods, we also implement the standard MsFEM in Petrov--Galerkin formulation which is denoted as MsPGM, and the MsPGM which uses the classical oversampling multiscale basis (OMsPGM).  We also show the results of the traditional linear finite element method (FEM)  and discontinuous finite element method (DFEM) on the corresponding coarse grid to get a feeling for the accuracy of the multiscale methods. All numerical experiments are designed to show better performance of MsDPGM than the other MsPGMs.}

{In all tests, for simplicity, we use the standard triangulation which is constructed by first dividing the domain $\Om$ into sub-squares of equal length $h$ and then connecting the lower-left and the upper-right vertices of each sub-square. For any coarse-grid element $K\in\T_{h}$ we put its macro-element
$S(K)$ in such a way that their barycenters are coincide and
their corresponding edges are parallel. The length of the horizontal and vertical edges of $S(K)$ is four times of the corresponding length of the edges of $K$.
We assume that all right-angle sides of $S(K), K\in\T_{h}$ have the same length denoted by $h_S$.
Recall the definition of the $d=\text{min}_{K\in \T_h}d_{K}$, define
 \begin{align}\label{etd}
 \tilde d=(h_S-h)/3.
 \end{align}
 It is clear that $d\eqsim\tilde d$. See Figure~\ref{figsample1} for an illustration.

\begin{figure}[htp]
  \centering
\begin{picture}(200,100)(0,0)
  \put(8,0){\line(1,0){96}}
  \put(8,0){\line(1,1){96}}
  \put(104,0){\line(0,1){96}}

  \put(56,24){\line(1,0){24}}
  \put(56,24){\line(1,1){24}}
  \put(80,24){\line(0,1){24}}

    \put(118,96){\line(1,0){96}}
  \put(118,0){\line(1,1){96}}
  \put(118,0){\line(0,1){96}}

  \put(142,72){\line(1,0){24}}
  \put(142,48){\line(1,1){24}}
  \put(142,48){\line(0,1){24}}
    \put(68,30){$K$}
    \put(78,60){$S(K)$}
    \put(146,62){$K$}
    \put(120,80){$S(K)$}

    \put(72,0){\line(0,1){6}}
    \put(72,18){\line(0,1){6}}
    \put(68,7){$\tilde d$}

    \put(8,0){\line(0,-1){6}}
    \put(104,0){\line(0,-1){6}}
    \put(8,-3){\line(1,0){36}}
    \put(104,-3){\line(-1,0){36}}
    \put(49,-8){$h_S$}

    \put(80,24){\line(1,0){6}}
    \put(80,48){\line(1,0){6}}
    \put(83,24){\line(0,1){6}}
    \put(83,48){\line(0,-1){6}}
    \put(81,32){$h_K$}

\end{picture}
\caption{The element $K$ and its oversampling element $S(K)$: lower-right elements (left) and upper-left elements (right). }
\label{figsample1}
\end{figure}
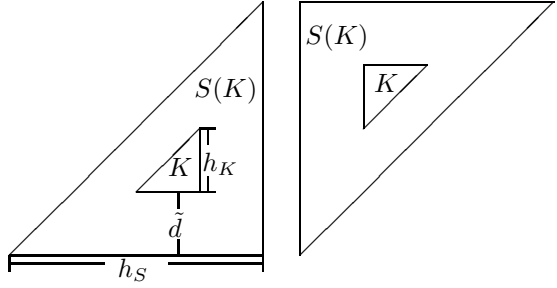

In all of these computations, we have used finely resolved numerical solutions obtained using the traditional linear finite element method with mesh size $h_f=1/4096$ as the ``exact'' solutions which are denoted as $u_e$.}
Denoting $u_h$ as the numerical solutions computed by the methods considered in this section,
we measure the relative error in $L^2$, $L^\infty$ and energy norms as following:
\[
\frac{\norm{u_h-u_e}_{L^2}}{\norm{u_e}_{L^2}},\,\frac{\norm{u_h-u_e}_{L^\infty}}{\norm{u_e}_{L^\infty}},\,\frac{\norm{u_h-u_e}_{1,h}}{\norm{u_e}_{1,h}},
\]
where
\[
||v||_{1,h} :=\left(\sum_{K\in\T_h}\|(\mathbf{a}^\ep)^{1/2}\na v\|_{L^2(K)}^2 +\norm{v}_{L^2(\Om)}^2\right)^{1/2}.
\]
In all tests, the coefficient $\mathbf{a}^\ep$ is chosen as the form $\mathbf{a}^\ep=a^\ep I$ where $a^\ep$ is a scalar function and $I$ is the 2 by 2 identity matrix.
\subsection{Application to elliptic problems with highly oscillating coefficients}
We first consider the model problem (\ref{eproblem}) in the squared domain
$\Om=[0,1]\times[0,1]$. Assume that $f=1$ and the coefficient
$\mathbf{a}^\ep(x_1,x_2)$ has
 the following periodic form:
\begin{equation}\label{coef1}
\mathbf{a}^\ep(x_1,x_2)=\frac{2+1.8\sin(2\pi x_1/\epsilon)}{2+1.8\cos(2\pi x_2/\epsilon)}
+\frac{2+1.8\sin(2\pi x_2/\epsilon)}{2+1.8\sin(2\pi x_1/\epsilon)},
\end{equation}
where we fix $\ep = 1/100$.

In this test, we choose $h=1/32$ and report errors in the $L^2 ,
L^\infty $ and energy norms in Table~\ref{Table:1}. { We can see that
the MsDPGM and MsDFEM give more accurate results than the other
multiscale methods considered here,} while the FEM and DFEM give
worse approximations to the gradient of solution. { We also compare
the CPU time $T_1$ and $T_2$ spent by the MsDFEM and MsDPGM to show the good performance of MsDPGM in computational complexity, where $T_1$ is
the CPU time of assembling the stiffness matrix, and $T_2$ is the
CPU time of solving the discrete system of algebraic equations. We can observe that the CPU time $T_1$ of our MsDPGM for assembling the
stiffness matrix is shorter than that of MsDFEM since the
Petrov--Galerkin method can decrease the computational complexity.}

\begin{table}[htp]
\caption{Compare different methods to show the accuracy of MsDPGM in
periodic case given by \eqref{coef1}. $\rho=\ep=1/100, \tilde
d=h=1/32,\ga_0=20$. }\label{Table:1}
\begin{center}
\begin{tabular}{|c|c|c|c|c|c|} \hline
 \multirow{2}{*}{Relative error}  &  \multirow{2}{*}{$L^2$} & \multirow{2}{*}{$L^\infty$} &\multirow{2}{*}{Energy norm}& \multicolumn{2}{|c|}{CPU time(s)} \\\cline{5-6}
 & & & & $T_1$ & $T_2$ \\\hline
   FEM           &  0.1150e-00 & 0.2311e-00 & 0.8790e-00  & -- & --\\\hline
   DFEM           &  0.2667e-00 & 0.2634e-00 & 0.5498e-00 & -- & --\\\hline
   MsPGM       &  0.7448e-01 & 0.7342e-01 & 0.2929e-00 &--& --\\\hline
   OMsPGM           &  0.1430e-01 & 0.1521e-01 & 0.1641e-00 & -- &--\\\hline
   MsDFEM        &  0.1007e-01 & 0.1029e-01 & 0.1629e-00 & 1.300& 0.028\\\hline
   MsDPGM        &  0.1266e-01 & 0.1395e-01 & 0.1631e-00 &1.119& 0.027
\\\hline
\end{tabular}
\end{center}
\end{table}

{ Secondly, we do an experiment to study how the penalty parameter $\gamma_0$ affects the errors. We fix $\rho=\ep=1/100, \tilde d=h=1/32$ and choose a series of $\gamma_0$ in the test. The result is shown in Table~\ref{Table:11}. We observe that as $\gamma_0$ goes larger, the relative error is close to the error of the OMsPGM. It seems that MsDPGM converges to OMsPGM as the penalty parameter $\gamma_0$ goes to infinity (cf. \cite{LN2000}).
\begin{table}[htp]
\caption{Convergence with respect to $\gamma_0$. $\rho=\ep=1/100, \tilde d=h=1/32$. }\label{Table:11}
\begin{center}
\begin{tabular}{|c|c|c|c|} \hline
 {Relative Error}  &  $L^2$ & $L^\infty$ &Energy norm \\\hline
   $\gamma_0=10$   &  0.1100e-01 & 0.1266e-01 & 0.1637e-00\\\hline
   $\gamma_0=20$    &  0.1266e-01 & 0.1395e-01 & 0.1631e-00\\\hline
   $\gamma_0=100$    &  0.1397e-01 & 0.1496e-01 & 0.1638e-00\\\hline
   $\gamma_0=1000$    &  0.1426e-01 & 0.1519e-01 & 0.1641e-00\\\hline
   $\gamma_0=10000$    &  0.1429e-01 & 0.1521e-01 & 0.1641e-00\\\hline
   \end{tabular}
\end{center}
\end{table}
}

The third numerical experiment is to show the mesh size $h$
plays a role as that describing in Theorem~\ref{energeerror}. We
fix $\tilde d= 1/32$ and $\ep=1/100$. Four kinds of mesh size
are chosen: $h=1/64, 1/32, 1/16, 1/8$. The results are shown in
Table~\ref{Table:2}. Relative error in energy norm against the mesh
size $h$ is clearly shown in Figure~\ref{errorh}. It is easy to see
that as $h$ goes larger, the relative error in energy norm goes
larger, which is in agreement with the theoretical results in
Theorem~\ref{energeerror}. { We remark that the classical MsFEM suffers from the resonance error since the $H^1$--error estimate has the term $\epsilon/h$ due to the nonconforming error (see \cite{HWC}). But for MsDPGM, the error estimate in Theorem~\ref{energeerror}, and the numerical results in Table~\ref{Table:2} and Figure~\ref{errorh} show that the resonance error has been removed completely.}

\begin{table}[htp]
\caption{Error with respect to $h$. $\rho=\ep=1/100, \tilde
d=1/32,\ga_0=20$. }\label{Table:2}
\begin{center}
\begin{tabular}{|c|c|c|c|} \hline
 {Relative error}  &  $L^2$ & $L^\infty$ &Energy norm \\\hline
   $h=1/64$   &  0.1371e-01 & 0.1474e-01 & 0.1593e-00\\\hline
   $h=1/32$   &  0.1266e-01 & 0.1395e-01 & 0.1631e-00\\\hline
   $h=1/16$    &  0.1948e-01 & 0.2532e-01 & 0.1870e-00\\\hline
   $h=1/8$    &  0.5210e-01 & 0.8248e-01 & 0.2620e-00\\\hline
   \end{tabular}
\end{center}
\end{table}

\begin{figure}[htp]
\centerline{\includegraphics[scale=0.51]{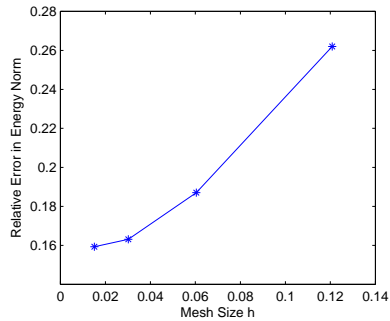}}
\caption{\label{errorh}Relative error with respect to $h$.}
\end{figure}

{\subsection{Affection of the size of the oversampling patches}

In this subsection we study how the size of oversampling elements affects the error.}
The experiment to verify the inequality \eqref{eq2} for the model example with coefficient \eqref{coef1} has been done in \cite{SDW2015}. The figures have been shown that $\norm{\na \eta^j}_{L^\infty(K)}\cdot d_{K}$ are bounded by a constant which is consistent with \eqref{eq2} (see Figure~5 in \cite{SDW2015}).

{ The following numerical experiment is to show how the oversampling size affects the error. Recalling the requirement of the oversampling size $d_K \geq \delta_0 h_K$, we show the relative oversampling size $\delta_0$ against the error. Note the distance $d=\text {min}_{K\in \T_h}d_K$. Here it is equivalent to $d \geq \delta_0 h$.}
{ We set $\rho=\ep=1/100, h=1/32$.
The result is shown in Table~\ref{d}. We can see that as $\delta_0$ (equivalently $\tilde d$) goes larger, the relative error in energy norm goes smaller, which is coincided with the theoretical results in Theorem~\ref{energeerror}. }We also notice that when $\tilde d$ is close to $\sqrt{\eps}$, the errors begin to decrease very slowly. Recall that there is a homogenization error $\sqrt{\eps}$ in the error estimate \eqref{energyest}. We think that when $d$ is large enough, $\sqrt{\eps}$ becomes the dominated error instead of $\eps/d$.

\begin{table}[htp]
\caption{Error with respect to $\delta_0$. $\rho=\ep=1/100, h=1/32,\ga_0=20$. }\label{d}
\begin{center}
\begin{tabular}{|c|c|c|c|} \hline
{Relative error}  &  $L^2$ & $L^\infty$ &Energy norm \\\hline
   $\delta_0=1/32$    &  0.4304e-01 & 0.4423e-01 & 0.2184e-00\\\hline
   $\delta_0=1/16$    &  0.2924e-01 & 0.3172e-01 & 0.1893e-00\\\hline
   $\delta_0=1/8$    &  0.1969e-01 & 0.2156e-01 & 0.1728e-00\\\hline
   $\delta_0=1/4$    &  0.1790e-01 & 0.2372e-01 & 0.1653e-00\\\hline
   $\delta_0=1/2$    &  0.1531e-01 & 0.1766e-01 & 0.1642e-00\\\hline
   $\delta_0=1$    &  0.1266e-01 & 0.1295e-01 & 0.1631e-00\\\hline
   $\delta_0=2$    &  0.1197e-01 & 0.1398e-01 & 0.1631e-00\\\hline
   \end{tabular}
\end{center}
\end{table}

\subsection{Application to multiscale problems on L--shape domain}
We consider the multiscale problem on the L--shaped domain of Figure \ref{perms1} with Dirichlet boundary condition so chosen that the true solution is $u=r^{\frac{1}{3}}\sin(2\theta/3)$ in polar coordinates. It is known that the solution has the singular behavior around reentrant corners. So the classical finite element method fails to provide satisfactory result.

\begin{figure}[htp]
\centerline{\includegraphics[scale=0.56]{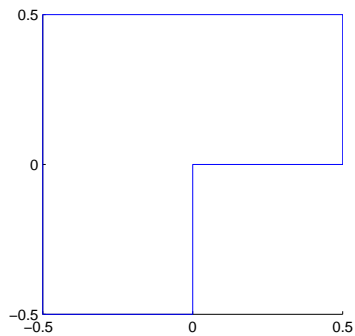}}
\caption{\label{perms1}The L--shape domain.}
\end{figure}

Firstly, we simulate the problem with coefficient given by \eqref{coef1}. We fix $\ep=1/100$ and choose $h=1/16$.
The relative error is shown in Table \ref{Table:5}. { We observe that both MsDPGM and MsDFEM give better approximation than the other MsPG methods.}

\begin{table}[htp]
\caption{Relative errors in the $L^2$ , $L^\infty$ and energy norm
for the L--shaped problem with periodic coefficient (\ref{coef1}).
$\rho=\ep=1/100, \tilde d=h=1/16,\ga_0=20$. }\label{Table:5}
\begin{center}
\begin{tabular}{|c|c|c|c|} \hline
 {Relative error}  &  $L^2$ & $L^\infty$ &Energy norm \\\hline
   MsPGM    &  0.7765e-02 & 0.3635e-01 & 0.2014e-00\\\hline
  OMsPGM   &  0.6285e-02 & 0.3277e-01 & 0.1035e-00\\\hline
  MsDFEM   &  0.3903e-02 & 0.2244e-01 & 0.9260e-01\\\hline
  MsDPGM    &  0.4654e-02 & 0.2299e-01 & 0.9275e-01\\\hline
   \end{tabular}
\end{center}
\end{table}

Secondly, we simulate the problem with the random log-normal
permeability field $\mathbf{a}(x)$, which is generated by using the
moving ellipse average \cite{dur} with the variance of the logarithm
of the permeability $\sigma^2 =1.0$, and the correlation lengths
$l_1=l_2=0.01$ in $x_1$ and $x_2$ directions, respectively. One
realization of the resulting permeability field is depicted in
Figure \ref{perms2}, where $\frac{a_{\rm max} (x)}{a_{\rm min}
(x)}=2.9642e+003$. In this test, we set $\rho=h=1/16$ since there is
no explicit $\ep$ in the example. The result is shown in Table
\ref{Table:6}. We can see that MsDPGM gives a better approximation
than the other MsPG methods, while standard MsPGM gives the wrong
approximation to the gradient of solution.
\begin{figure}[htp]
\centerline{\includegraphics[scale=0.56]{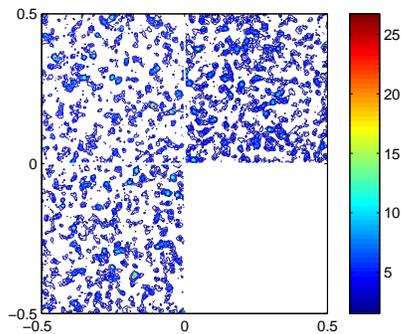}}
\caption{\label{perms2}The random log-normal permeability field
$\mathbf{a}(x)$. $\frac{a_{\rm max} (x)}{a_{\rm min} (x)}=2.9642e+003$.}
\end{figure}

\begin{table}[htp]
\caption{Relative errors in the $L^2$ , $L^\infty$ and energy norm
for the L--shaped problem with random coefficient $\sigma^2 =1.0$ and
$l_1=l_2=0.01$. $\tilde d=\rho=h=1/16,\ga_0=20$. }\label{Table:6}
\begin{center}
\begin{tabular}{|c|c|c|c|} \hline
 {Relative error}  &  $L^2$ & $L^\infty$ &Energy norm \\\hline
 MsPGM    &  0.9074e-00 & 0.1290e+01 & 0.6601e+02\\\hline
 OMsPGM   &  0.9307e-02 & 0.3851e-01 & 0.1428e-00\\\hline
 MsDFEM  &  0.6504e-02 & 0.3718e-01 & 0.9931e-01\\\hline
 MsDPGM   &  0.8587e-02 & 0.3810e-01 & 0.1013e-00\\\hline
   \end{tabular}
\end{center}
\end{table}

\section{Conclusion}
In this paper, we have proposed a new Petrov--Galerkin method based
on the discontinuous multiscale approximation space for the multiscale elliptic
problems. Under some
assumptions on the coefficients, we give the error analysis of our
method. The $H^1$--error is of the order
\begin{equation*}
    O\Big(\sqrt{\ep}+\frac{\ep}{d}+h+\frac{h^{3/2}}{\sqrt{\ep}}\Big),
\end{equation*}
which consists of the oversampling multiscale approximation error
and the error contributed by the penalty. Note that the unpleasant
resonance error does not appear. The reason is that our method uses
discontinuous piecewise linear functions as test functions, which is
only needed to estimate the interpolation error. Several numerical
experiments have { demonstrated} the efficiency of MsDPGM. 
{We also study the corresponding MsDFEM which coupling the classical oversampling multiscale basis with DGM. Our convergence analysis shows that MsDFEM can also eliminate the resonance error completely. That is the reason why MsDFEM is working as well as MsDPGM (even a little better). Furthermore, we can see that the CPU-time cost of MsDPGM for assembling the stiffness matrix is shorter than that of the MsDFEM due to its PG version. Therefore, we think that MsDPGM is a good choice when we need to take into consideration of the computational accuracy and the computer resource at the same time.}

{ We emphasize that the proposed method is not restrict to the periodic case. The numerical experiments show that it is applicable to the random coefficient case very well. However, with the classical oversampling multiscale basis function space introduced in \cite{HW}, the error estimate method is based on the classical homogenization theory, which needs the assumption that the oscillating coefficient is periodic. In the future work, we would like to combine the Petrov--Galerkin method with the new oversampling multiscale space \cite{Hp13} to consider the elliptic multiscale problems without any assumption on scale separation or periodicity. Besides, the introduced method may be inefficient for the multiscale problems which have some singularities, such as, the Dirac function singularities which stems from the simulation of steady flow transport through highly heterogeneous porous media driven by extraction wells \cite{CY2002}, or high-conductivity channels that connect the boundaries of coarse-grid blocks \cite{EGW2011}. To solve these problems, it needs some special definition of the multiscale basis functions around the channels such as the local spectral basis functions (see \cite{EGW2011}), or local refinement of the elements near the channels (see \cite{DW2014}). We will couple these techniques with the introduced method in our future work. Finally, we remark that Generalized Multiscale Finite Element method coupling DGM was explored in \cite{ya2013}. The computation is divided into two stages: offline and online. In the offline stage, they construct a reduced dimensional multiscale space
to be used for rapid computations in the online stage. In the online stage, they use the basis functions computed offline to solve the problem for current realization of the parameters. Similar to MsDPGM, in the online stage we can use the Petrov--Galerkin version of DGM to solve the problem with the basis functions computed offline, which leads to a kind of  Generalized Multiscale Discontinuous Petrov-Galerkin method. The difficulty is the choice of the test function space and the proof of inf-sup condition, which is worth studying.}

\section*{Acknowledgments}

The authors would like to thank the referees for their carefully
reading and constructive comments that improved the paper.

\appendix
\section{Proof of Theorem~\ref{h2homoerror}}\label{AA}

The following theorem plays an important role in our analysis (cf.
\cite{CH2002, CW2010}).

\begin{Theorem}\label{errorH1} Assume that $ u_0 \in H^2(\Om)\cap W^{1,\infty}(\Om)$. There exists a constant $C$ independent of $u_0, \ep, \Om$ such that
\begin{equation*}
\begin{split}
  &||u_{\ep}-u_1-\ep\theta_\ep||_{H^1(\Om)} \leq C \ep|u_0|_{H^2(\Om)},\\
  &||\ep\theta_\ep||_{H^1(\Om)} \leq  C \sqrt{\ep}|u_0|_{W^{1,\infty}(\Om)}+ C \ep|u_0|_{H^2(\Om)},
\end{split}
\end{equation*}
where $\theta_\ep$ denote the boundary corrector defined by
\begin{equation}
\begin{aligned}
       -\na\cdot(\mathbf{a}^\ep\na\theta_\ep) & = 0 &&  \mbox{in }\Om,\\
       \theta_\ep&=-\chi^j(x/\varepsilon)\frac{\pa
u_0(x)}{\pa x_j} &&  \mbox{on
       }\pa\Om.
\end{aligned}
\end{equation}
\end{Theorem}

We first estimate $ \abs{\ep\theta_{\ep}}_{H^2(\Om)} $.
\begin{lemma}\label{eh2homoerror}Assume that $ u_0 \in H^2(\Om)\cap W^{1,\infty}(\Om)$. Then the following estimate holds:
\begin{equation}\label{eh2homoerro}
    \abs{\ep\theta_{\ep}}_{H^2(\Om)}\ls \frac{1}{\sqrt{\ep}}\abs{u_0}_{W^{1,\infty}(\Om)}+\abs{u_0}_{H^2(\Om)}.
\end{equation}
\end{lemma}
\begin{proof}
We only consider the case where $n=2$. For $n=3$, the proof is similar.
Let $\xi \in C_{0}^{\infty}(\mathbf{R}^2)$ be the cut-off function such that $ 0 \leq \xi \leq 1, \xi=1$ in $\Om \setminus \Om_{\ep/2}, \xi=0 $ in $\Om_{\ep}$, and $\abs{\na \xi} \leq C/ \ep, \abs{\na^2 \xi} \leq C/ \ep^2$ in $\Om$, where $\Om_{\ep}:=\{x: \text{dist}\{x,\pa \Om\}\geq \ep \}$. Then
\begin{equation*}
   v=\theta_{\ep}+\xi(\chi^j \frac{\pa u_0}{\pa x_j})\in H_0^1(\Om)
\end{equation*}
satisfies
\begin{equation}
     -\na \cdot(\mathbf{a}^\ep\na v) =-\na \cdot(\mathbf{a}^\ep\na (\xi\chi^j \frac{\pa u_0}{\pa x_j}))  \qquad \text{ in } \Om,\qquad v|_{\partial \Om}=0.
\end{equation}
By use of Theorem 4.3.1.4 in \cite{Grisvard} , and together with Theorem \ref{errorH1}, we have
\begin{equation*}
\begin{split}
   \abs{v}_{H^2(\Om)}&\ls \frac{1}{\ep}\norml{v}{\Om}+\norml{\na \cdot(\mathbf{a}^\ep\na (\xi\chi^j \frac{\pa u_0}{\pa x_j}))}{\Om} \\ &\ls \frac{\sqrt{\ep}}{\ep^2}\abs{u_0}_{W^{1,\infty}(\Om)}+\frac{1}{\ep}\abs{u_0}_{H^2(\Om)},
   \end{split}
\end{equation*}
which implies
\begin{equation}\label{global1}
    \abs{\theta_{\ep}}_{H^2(\Om)}\ls \frac{\sqrt{\ep}}{\ep^2}\abs{u_0}_{W^{1,\infty}(\Om)}+\frac{1}{\ep}\abs{u_0}_{H^2(\Om)}.
\end{equation}
This completes the proof.
\end{proof}

\noindent{\bf Proof of Theorem~\ref{h2homoerror}.}
It is shown that, for any $\varphi\in H_0^1(\Om)$ (see \cite[p.550]{CH2002} or \cite[p.125]{CW2010}),
\begin{equation}\label{e20}
\begin{split}
 &\left(\mathbf{a}(x/\varepsilon)\nabla(u_\ep-u_1),\nabla
 \varphi\right)_{\Om}\\
 &\quad=(\mathbf{a}^*\nabla u_0, \nabla
 \varphi)_{\Om}-\left(\mathbf{a}(x/\varepsilon)\nabla\left(u_0+\varepsilon\chi^k\frac{\partial
 u_0}{\partial x_k}\right),\nabla\varphi \right)_{\Om}\\
 &\quad=\varepsilon\int_{\Om} a_{ij}(x/\varepsilon)\chi^k
 \frac{\partial^2 u_0}{\partial x_j\partial
 x_k}\frac{\partial\varphi}{\partial x_i}\dx-\varepsilon\int_{\Om}
 \alpha_{ij}^k(x/\varepsilon) \frac{\partial^2 u_0}{\partial
 x_j\partial x_k}\frac{\partial\varphi}{\partial x_i}\dx,
 \end{split}
\end{equation}
where
$\alpha^k(x/\varepsilon)=(\alpha_{ij}^k(x/\varepsilon))$ are skew-symmetric matrices which satisfy that (see \cite[p.6]{JKO})
\begin{equation*}
G_i^k(y)=\frac{\partial}{\partial y_j}(\alpha_{ij}^k(y)),\qquad
\int_Y \alpha_{ij}^k(y) \dy=0
\end{equation*}
with
\[
G_i^k=a_{ik}^*-a_{ij}\left(\delta_{kj}+\frac{\partial\chi^k}{\partial
y_j}\right).
\]
From \eqref{e20}, it follows that,
\[
\na\cdot\left(\mathbf{a}(x/\varepsilon)\nabla(u_\ep-u_1)\right)=\ep\frac{\pa}{\pa x_i}\left(a_{ij}(x/\varepsilon)\chi^k
 \frac{\partial^2 u_0}{\partial x_j\partial
 x_k}-\alpha_{ij}^k(x/\varepsilon) \frac{\partial^2 u_0}{\partial
 x_j\partial x_k}\right),
\]
which combines the definition of $\theta_\ep$ yield
\[
\na\cdot\left(\mathbf{a}(x/\varepsilon)\nabla(u_\ep-u_1-\ep \theta_{\ep})\right)=\ep\frac{\pa}{\pa x_i}\left(a_{ij}(x/\varepsilon)\chi^k
 \frac{\partial^2 u_0}{\partial x_j\partial
 x_k}-\alpha_{ij}^k(x/\varepsilon) \frac{\partial^2 u_0}{\partial
 x_j\partial x_k}\right).
\]
Thus, from Theorem 4.3.1.4 in \cite{Grisvard}, it follows that
\begin{equation}\label{global}
    \abs{u_\ep-u_1-\ep \theta_{\ep}}_{H^2(\Om)}\ls \frac{1}{\ep}\norml{u_\ep-u_1}{\Om}+\abs{u_0}_{H^2(\Om)}+\ep\abs{u_0}_{H^3(\Om)},
\end{equation}
which combing \eqref{eh2homoerro} and Theorem \ref{errorH1}, yield \eqref{h2homoestimate} immediately.
\endproof

{
\section{Theoretical Results of MsDFEM}\label{AB}

We give some theoretical results of MsDFEM here for convenience of the reader. Detailed analysis can be found in the first author's PHD thesis \cite{s2016}.
\begin{lemma}
We have
\begin{equation}\label{cont1}
    \abs{a(\tilde{u}_h,v_h)}\leq C\norm{\tilde{u}_h}_{E}\norm{v_h}_{E}  \qquad\forall\, \tilde{u}_h,v_h \in V_{h,dc}^{ms}.
\end{equation}
Further, let the assumptions of Lemma \ref{inverse1} be fulfilled and $\ga_0$ is large enough, then
\begin{equation}\label{coer1}
   a(v_{h}, v_{h})\geq \frac{1}{2}\norm{v_h}_{E}^2  \qquad\forall\, v_{h} \in V_{h,dc}^{ms}.
\end{equation}
Here
\begin{equation*}
 \begin{split}
 \norm{v}_{E}:=&\Big(\sum_{K\in\T_h}\int_K \mathbf{a}^\ep |\na v|^2\dx
 +\sum_{e\in\Ga_h\cup \pa \Om}\frac{\rho}{\gamma _0}\int_{e} \av{\mathbf{a}^\ep \na v\cdot\bn}^2\ds\\
 &\qquad\qquad+ \sum_{e\in\Ga_h\cup \pa \Om}\frac{\gamma _0}{\rho}\int_{e} \jm{v}^2\ds
 \Big)^{1/2}\qquad\forall v\in V_{h,dc}^{ms}.
 \end{split}
 \end{equation*}
\end{lemma}
Using the definition of the above norm, the Cauchy-Schwarz inequality and \eqref{trace}, Lemma~\ref{inverse1}, we can obtain \eqref{cont1} and \eqref{coer1} immediately. The proof is similar to Theorem~\ref{ccPG} and is omitted here.

\begin{Theorem}\label{energeerror112}Let $u_\ep$ be the solution of \eqref{eproblem}, and let $\tilde{u}_h$ be the numerical solution computed by MsDFEM defined in (\ref{eifem}). Assume that $u_0 \in H^3(\Om), f \in L^2(\Om)$, $\ep \ls h \ls d$,
and that the penalty parameter $\ga_0$ is large enough. Then there
exits a constant $\ga$ independent of $h$ and $\ep$ such that if
$\ep/h_K \leq \ga$ for all $K \in \T_h$, the following error
estimate holds:
\begin{equation}\label{energeerror111}
  \norm{u_\ep-\tilde{u}_{h}}_{E}
  \lesssim h+\frac{h^{3/2}}{\sqrt{\ep}}
  +\sqrt{\ep}+\frac{\ep}{d},
\end{equation}
where $d=\text {min}_{K\in \T_h}d_K$.
\end{Theorem}
\begin{proof}
By use of the Galerkin orthogonality of $a(\cdot,\cdot)$, we only need to estimate the interpolation error.

Take $v_{h}$ as $\psi_h$ (see \eqref{intplMsFEM}). The following two estimates of the error have been shown in the proof of Theorem~\ref{energeerror}:
\begin{equation}\label{result1}
    \Big(\sum_{K \in \T_h}\norml{ (\mathbf{a}^\ep)^{1/2} \na(u_\ep-v_h)}{K}^2\Big)^{\frac{1}{2}}\lesssim h|u_0|_{H^2(\Om)}
    +\Big(\sqrt{\ep}+\frac{\ep}{d}\Big)|u_0|_{W^{1,\infty}(\Om)},
\end{equation}
and
\begin{equation}\label{I1}
\begin{split}
    &\sum_{e \in \Ga_h\cup \pa \Om}\frac{\ep}{\gamma_0}\norml{ \av {\mathbf{a}^\ep\na (u_\ep-v_h)\cdot \bn}}{e}^2\\
    &\qquad\qquad\lesssim h^2|u_0|^2_{H^2(\Om)}+\ep |u_0|^2_{W^{1,\infty}(\Om)}+\ep^4|u_0|^2_{H^3(\Om)}.
    \end{split}
\end{equation}

It remains to consider the term
$ \sum_{e \in \Ga_h\cup \pa \Om}\frac{\gamma_0}{\ep}\norml{\jm{u_\ep-v_h}}{e}^2 $.
Noting that $\jm{u_{\ep}}=\jm{u_1}=0$, then by use of the trace inequality \eqref{trace} and Lemma \ref{lmultiapp}, we have
\begin{equation}\label{result2}
    \begin{split}
     \sum_{e \in \Ga_h\cup \pa \Om}&
     \frac{\gamma_0}{\ep}\norml{\jm{u_\ep-v_h}}{e}^2 \lesssim
     \sum_{e \in \Ga_h\cup \pa \Om}
     \frac{\gamma_0}{\ep}\norml{\jm{u_1-v_h}}{e}^2\\
    & \lesssim  \ep^{-1} h^{-1}\sum_{K \in \T_h}\norml{u_1-v_h}{K}^2\\
     & +\ep^{-1} \Big(\sum_{K \in \T_h}\norml{u_1-v_h}{K}^2\Big)^{1/2}\Big(\sum_{K \in \T_h}\norml{\na(u_1-v_h)}{K}^2\Big)^{1/2}\\
    & \lesssim   \frac{h^3}{\ep}|u_0|_{H^2(\Om)}^2+\ep |u_0|^2_{W^{1,\infty}(\Om)},
     \end{split}
\end{equation}
where we have used the assumption $\ep \ls h \ls d$ and the Young's inequality to derive the above inequality.

Hence, from \eqref{result1}, \eqref{I1} and \eqref{result2}, it follows \eqref{energeerror111}
immediately. This completes the proof.
\end{proof}}

\end{document}